\documentclass[reqno]{amsproc}
\usepackage{amssymb}
\usepackage{amsmath}
\usepackage{mathdots}
\usepackage{amsbsy}
\usepackage{amscd}
\usepackage{amsthm,longtable}

\usepackage{url}
\usepackage{xcolor}
\usepackage{amsmath,amssymb,amsthm,amsfonts,longtable}
\usepackage[english]{babel}
\usepackage{tikz-cd}
\usetikzlibrary{cd}
\usetikzlibrary{decorations.markings}
\tikzset{negated/.style={
		decoration={markings,
			mark= at position 0.5 with {
				\node[transform shape] (tempnode) {$\times$};
			}
		},
		postaction={decorate}
	}
}
\usepackage{array}
\usepackage[colorlinks,citecolor=red,urlcolor=blue,bookmarks=false,hypertexnames=true]{hyperref}

\newtheorem{definition}{Definition}[section]
\newtheorem{lemma}[definition]{Lemma}

\newtheorem{theorem}[definition]{Theorem}
\newtheorem{corollary}[definition]{Corollary}
\newtheorem{remark}[definition]{Remark}

\newcommand{\ind}{\mathord\uparrow} 
\newcommand{\Irr}{\textnormal{Irr}}
\newcommand{\cd}{\textnormal{cd}}
\newcommand{\nl}{\textnormal{nl}}
\newcommand{\lin}{\textnormal{lin}}
\newcommand{\Core}{\textnormal{Core}}

\newcommand{\cl}{\textnormal{cl}}
\newcommand{\cod}{\textnormal{cod}}

\title[Character Codegree]
{Codegrees of Irreducible Characters of VZ and Camina $p$-Groups} 

\author{Ayush Udeep}
\address{SRM Institute of Science and Technology Kattankulathur, Chennai-603203, Tamil Nadu, India}
\email{udeepayush@gmail.com}

\begin{document}
	\maketitle
	\begin{abstract}
		The character codegree of an irreducible character  of a finite group $G$ is given by the index of its kernel in $G$ upon the character degree. We compute the codegrees of irreducible characters of VZ and Camina $p$-groups, and also obtain the character codegrees set of $p$-groups of order $\leq p^5$, where $p$ is an odd prime.
	\end{abstract}
	
	\section{Introduction}
	In this article, $G$ denotes a finite $p$-group, where $p$ denotes a prime. The character codegree of an irreducible character $\chi$ of $G$ is the index of the kernel of $\chi$ in $G$ upon the degree of $\chi$. We denote the character codegree of $\chi$ by $\cod(\chi)$, and the set of character codegrees of all irreducible characters of $G$, all linear characters of $G$ and all nonlinear irreducible characters of $G$ by $\cod(G)$, $\cod_{\lin}(G)$ and $\cod_{\nl}(G)$, respectively.
	Qian et al. \cite{QWW2007} initiated the study of the irreducible character codegrees of $G$. In recent past, various researchers have obtained interesting structural results for various classes of non-abelian $p$-groups by studying their character codegree sets. As it turns out, certain versions of character codegrees of known results for character degrees hold for some classes of $p$-groups. For example, Du and Lewis \cite{DL2016} proved that the nilpotence class of $G$ is bounded in terms of the largest codegree for an irreducible character of $G$. They also proved that a $p$-group with exactly two codegrees is elementary abelian,
	and a $p$-group with exactly three codegrees has nilpotence class at most 2. In \cite{CL2020}, Croome and Lewis studied p-groups with exactly four codegrees. In \cite{CL2020 maximal}, Croome and Lewis proved that if a non-abelian group $G$ has order $p^n$ and $\cod(G)$ contains
	every power of $p$ up to $p^{n-1}$, then the nilpotence class of G can either be 2 or $n-1$. Moret\'{o} \cite{M2022} proved that the nilpotence class of a p-group is bounded in terms of the number of character degrees and the number of character codegrees.
		On the other hand, there are certain properties of character degrees that character codegrees do not respect. For example, while isoclinism preserves the set of irreducible character degrees of $G$, denoted $\cd(G)$, it does not preserve $\cod(G)$. For example, consider groups $\phi_{2}(211)a$ and $\phi_{2}(211)b$ of order $p^4$ $(p>3)$ belonging to the isoclinic family $\Phi_2$ (see \cite{J1980} for the classification of $p$-groups of order $p^4~(p>3)$). We have $\cod(\phi_{2}(211)a) = \{1, p, p^2\}$ whereas $\cod(\phi_{2}(211)b) = \{1, p, p^3\}$ (see Table \ref{t:p34}). 
		
		In a different direction, Prajapati and Udeep \cite{PU2024} established a relation between the minimal faithful permutation representation degree of a $p$-group of odd order and its character codegrees. All these instances motivate us to explore other relevant properties of the character codegrees for various classes of finite $p$-groups.
		In this article, we obtain the character codegrees for certain interesting classes of finite $p$-groups. In Section \ref{sec:abelian and VZ}, we compute this quantity for the irreducible characters of VZ and Camina $p$-groups. We begin by proving that for any $p$-group $G$, $\cod_{\lin}(G) = \{ p^{i} ~:~ i= 0, 1, \ldots,\log_{p}(\exp(G/G')) \}$ (see Corollary \ref{cor:linearcod}). A group is called a VZ-group if all its nonlinear irreducible complex characters vanish off the center. VZ-groups were studied in \cite{FAM2001, L2009 vanishing, L2009}.  We obtain an algorithm for the computation of $\cod(G)$ when $G$ is any general VZ-group (see Remark \ref{remark:VZgroups}). We observe that the problem of computation of $\cod(G)$, in this case, reduces to finding the exponent of $G/G'$ and obtaining the set
	\[ \left\{  \left| N \right| ~:~ N \lneq Z(G) \text{ such that } G' \nleq N \text{ and } \frac{Z(G)}{N} \text{ is cyclic} \right\}. \]
	\noindent We also compute the exact character codegrees for certain sub-classes of VZ $p$-groups. In fact, we prove the following.
	\begin{theorem} \label{thm: VZ p-group}
		Let $G$ be a VZ $p$-group. 
		\begin{enumerate}
			\item [(i)] If $Z(G)$ is cyclic, then $\cod(G) = \{ p^i, (|G||Z(G)|)^{1/2}~:~ i = 0, 1, \ldots, \log_{p}(\exp(G/G')) \}$.
			\item [(ii)] If $Z(G)$ is elementary abelian, then $\cod(G) = \{ p^i, p\cdot|G/Z(G)|^{1/2}~:~ i = 0, 1, \ldots, \log_{p}(\exp(G/G')) \}$.
			\item [(iii)] If $d(G') = d(Z(G))$, then \[ \cod(G) = \{ p^i, p^{l_j}\cdot|G/Z(G)|^{1/2}~:~ i = 0, 1, \ldots, \log_{p}(\exp(G/G')), j = 1,2,\ldots, r \}, \] where $Z(G) \cong C_{p^{l_1}} \times C_{p^{l_2}} \times \cdots \times C_{p^{l_r}} ~ (l_1\geq l_2\geq \cdots \geq l_r)$.
			\item [(iv)] If $d(G') < d(Z(G)) = 2$, then either \[ \cod(G) = \{ p^i, p^{j}\cdot|G/Z(G)|^{1/2}~:~ i = 0, 1, \ldots, \log_{p}(\exp(G/G')), j = l_2, l_2+1,\ldots, l_1 \}, \] or
			\[ \cod(G) = \{ p^i, \exp(Z(G))\cdot|G/Z(G)|^{1/2}~:~ i = 0, 1, \ldots, \log_{p}(\exp(G/G')) \} \]
			where $Z(G) \cong C_{p^{l_1}} \times C_{p^{l_2}}  ~ (l_1\geq l_2)$.
		\end{enumerate}
	\end{theorem}
	As a corollary of Theorem \ref{thm: VZ p-group}, we prove that
	for an extraspecial $p$-group $G$, $\cod(G) = \{ 1, p, \left( p\cdot |G| \right)^{1/2} \}$ (see Corollary \ref{cor:extraspecial p-group}). We also obtain the character codegrees set for Camina $p$-groups.  A $p$-group $G$ is called a Camina $p$-group if the conjugacy class of $g$ is $gG{}'$ for all $g\in G\setminus G{}'$ \cite{DS1996, L2014}. Dark and Scoppola \cite{DS1996} proved that Camina $p$-groups are nilpotent of class at most $3$. For any Camina $p$-group $G$, we prove that if the nilpotency class of $G$ is 2, then $\cod(G) = \{ 1, p, p\cdot|G/Z(G)|^{1/2} \}$, otherwise, $\cod(G) = \left\{ 1, p, p \cdot |G/G'|^{1/2}, p \cdot |G/Z(G)|^{1/2}  \right\}$ (see Theorem \ref{thm:CodG Camina pgroup}). It is interesting to note that the extraspecial $p$-groups and Camina $p$-groups of class 2 form classes of $p$-groups with exactly three character codegrees, while Camina $p$-groups of class 3 form a class of $p$-groups with exactly four character codegrees.
	
	In Section \ref{sec:groups less than p6}, we obtain the character codegrees sets for all $p$-groups of order $\leq p^5$, where $p$ is an odd prime. The $p$-groups of order $\leq p^5$, where $p\geq 3$, are classified in certain isoclinic families, and listed in \cite{J1980}.
	Since isoclinism does not preserve $\cod(G)$, the task of computing the character codegrees set of $p$-groups collectively becomes a little difficult. However, we do obtain general results for certain classes of groups of order $p^5$ $(p>3)$ in Lemma \ref{lemma:p5 cl3 cdG1,p} and \ref{lemma:p5 cl3 cdG1,p,p^2}. By utilizing them along with Theorem \ref{thm: VZ p-group} and Corollary \ref{cor:extraspecial p-group}, we compute the character codegrees set for all $p$-groups of order less than $p^6$ $(p>3)$ in Section \ref{sec:groups less than p6}, and list them in Table \ref{t:VZ groups phi2 phi5} and \ref{t:non VZ groups}. We also construct {\sc Magma} \cite{MAGMA} code {\tt CharacterCodegree} to obtain $\cod(G)$ for a given group $G$. We utilize it to verify the data of Tables \ref{t:p34}, \ref{t:VZ groups phi2 phi5} and \ref{t:non VZ groups} for small primes $\geq 5$, and to compute $\cod(G)$ when $G$ is a non-abelian group of order $3^n$ for $n=3,4,5$.
	We summarize the results of Section \ref{sec:groups less than p6} in Theorem \ref{thm:codG total p^5}.
	
	Before we start, here we mention a list of notations (other than those already mentioned above) that we use throughout the article. For a finite group $G$, \\
	
	\begin{tabular}{cl}
		$G'$ & the commutator subgroup of $G$\\
		$\cl(G)$ & the nilpotency class of $G$\\
		$\exp(G)$ & the exponent of $G$\\
		$\Core_{G}(H)$ & the core of $H$ in $G$, for $H\leq G$\\
		$\Irr(G)$ & the set of irreducible complex characters of $G$\\
		$\lin(G)$ & the set of linear characters of $G$\\
		$\nl(G)$ & the set of non-linear irreducible characters of $G$\\
		$\cd(G)$ & the character degree set of $G$, i.e., $\{ \chi(1) ~|~ \chi \in \Irr(G) \}$\\
		$\Irr(G|N)$ & $\Irr(G) \setminus \Irr(G/N)$, where $N$ is a normal subgroup of $G$\\
		$\nu$ & the smallest positive integer which is a quadratic non-residue modulo $p$
	\end{tabular}
	
	
	\section{Character Codegrees of VZ $p$-groups and Camina $p$-groups} \label{sec:abelian and VZ}
	
	We begin with the following lemma.
	
	\begin{lemma} \label{lemma:normalsubgroupabelian}
		Let $A$ be an abelian $p$-group. If $M$ is a subgroup of $A$ with $A/M$ cyclic, then $|M| \geq |A|/\exp(A)$. Further, for each $b$, where $0\leq b \leq \log_{p}(\exp(A))$, there exists a subgroup $N$ of $A$ such that $A/N \cong C_{p^b}$.
	\end{lemma}
	\begin{proof}
		 If $|M| < |A|/\exp(A)$, then $|A/M| > \exp(A)$, and hence, $A/M$ cannot be cyclic. Thus, if $A/M$ is cyclic, then $|M| \geq |A|/\exp(A)$.
		To prove the second part of the lemma, suppose
		\[ A = \langle a_1, a_2, \ldots, a_s \rangle \cong C_{p^{a}} \times C_{p^{r_2}} \times \cdots \times C_{p^{r_s}}, \]
		where $r_i$'s are positive integers such that $a = r_1\geq r_2\geq \cdots \geq r_s$, that is, $a = \exp(A)$. To find a subgroup $N$ of $A$ of order $|A|/p^b$ so that $A/N$ is cyclic, for each $0\leq b \leq a$, we use the following algorithm:
		\begin{enumerate}
			\item Let $X = [a = r_1, r_2, \ldots, r_s]$ be an ordered set, i.e., $X[j] = r_j$ for some $1\leq j \leq s$.
			\item Take $i=s$ (note that $i$ ranges from $1$ to $s$).
			\item If $X[i] \geq b$, then 
			\[ N = \left\langle a_1, a_2, \ldots, a_{s-1}, a_{s}^{p^{b}} \right\rangle. \]
			\item Otherwise, if $X[i] < b$, then $i = i-1$.
			\item Continue till $i=1$.
		\end{enumerate}
		Since $b\leq a$, existence of such an $N$ is guaranteed after some steps of the algorithm.
	\end{proof}
	We derive the following corollary with the help of Lemma \ref{lemma:normalsubgroupabelian}.
	\begin{corollary}\label{corollary:abeliancod}
		Let $A$ be an abelian $p$-group.
		Then $\cod(A) = \{ p^{i} ~|~ i= 0, 1, \ldots,\log_{p}(\exp(A)) \}$.
	\end{corollary}
	\begin{proof}
		Since $A$ is abelian, all the irreducible characters are linear. Then $A/\ker(\chi)$ is cyclic, for each $\chi \in \Irr(A)$. Hence we consider those normal subgroups $N$ of $A$ such that $A/N$ is cyclic. Let $\exp(A) = p^a$. By Lemma \ref{lemma:normalsubgroupabelian}, $|\ker(\chi)| \geq |A|/p^a$ so that $\cod(\chi) \leq p^a$, for each $\chi$. Further, for each $b$ $(0\leq b \leq a)$, there exists $\chi \in \Irr(A)$ such that $|\ker(\chi)| = |A|/p^b$. That is, for each $b$ $(0\leq b \leq a)$, there exists $\chi \in \Irr(A)$ so that $\cod(\chi) = |A/\ker(\chi)| = p^b$. Therefore, $\cod(A) = \{ p^{b} ~|~ 0\leq b \leq a \} = \{ p^{i} ~|~ i= 0, 1, \ldots,\log_{p}(\exp(A)) \}$.
	\end{proof}
	
	From Corollary \ref{corollary:abeliancod}, we deduce the codegrees of the linear characters of any $p$-group.
	\begin{corollary} \label{cor:linearcod}
		Let $G$ be a $p$-group. Then $\cod_{\lin}(G) = \{ p^{i} ~|~ i= 0, 1, \ldots,\log_{p}(\exp(G/G')) \}$.
	\end{corollary}
	
	A group $G$ is called a VZ $p$-group if all its nonlinear irreducible complex characters vanish off the center.  Hence, $\nl(G)=\{|G/Z(G)|^{1/2}\lambda ~|~ \lambda \in \Irr(Z(G)) \textnormal{ such that } G{}' \nleq \ker(\lambda) \}$. A VZ $p$-group is nilpotent of class 2, and hence, the commutator subgroup of the group is contained in the center (see \cite{L2009}). In fact, $G$ is a VZ-group, if and only if 	$\cd(G)=\{1, |G/Z(G)|^{1/2} \}$ (see \cite[Theorem A]{FAM2001}). Further, from \cite[Lemma 2.4]{L2009 camina groups}, $G/Z(G)$ and $G'$ are elementary abelian	$p$-groups.  In Remark \ref{remark:VZgroups}, we obtain an algorithm for the computation of $\cod(G)$ when $G$ is a VZ $p$-group.
	
	\begin{remark} \label{remark:VZgroups}
		(Algorithm for the computation of $\cod(G)$ when $G$ is a VZ $p$-group)\\	\textnormal{Let $G$ be a VZ $p$-group. By Corollary \ref{cor:linearcod}, $\cod_{\lin}(G) = \{ p^i ~:~ i = 0, 1, \ldots, \log_{p}(\exp(G/G')) \}$. Thus, $\cod_{\lin}(G)$ depends on the exponent of $G/G'$. Since $\chi(1) = |G/Z(G)|^{1/2}$ for all $\chi\in \nl(G)$, computation of $\cod_{\nl}(G)$ depends solely on the possibilities of the order of $\ker(\chi)$.
			That is, for each $\chi \in \nl(G)$, \[ \cod(\chi) = \frac{|G/\ker(\chi)|}{\chi(1)} = \frac{|G/Z(G)||Z(G)/\ker(\chi)|}{|G/Z(G)|^{1/2}} = |G/Z(G)|^{1/2} \left| \frac{Z(G)}{\ker(\chi)} \right|. \]
			Clearly, $\ker(\chi) \lneq Z(G)$ for all $\chi\in \nl(G)$. Thus,
			\[ \cod_{\nl}(G) = \left\{ |G/Z(G)|^{1/2} \left| \frac{Z(G)}{N} \right| ~:~ N \lneq Z(G) \text{ such that } G' \nleq N \text{ and } \frac{Z(G)}{N} \text{ is cyclic} \right\}.  \] 
			Finally, $\cod(G) = \cod_{\lin}(G) \cup \cod_{\nl}(G)$.
		}
	\end{remark}
	By Remark \ref{remark:VZgroups}, it is clear that we may need more information about the group structure to compute the exact $\cod(G)$ for a general VZ-group.
	Now, we prove Theorem \ref{thm: VZ p-group}, where we utilize the properties of some specific classes of VZ $p$-groups to obtain their respective character codegrees sets.\\
	
	{\sc Proof of Theorem \ref{thm: VZ p-group}}.
	Let $G$ be a VZ $p$-group. By Corollary \ref{cor:linearcod}, $\cod_{\lin}(G) = \{ p^i ~:~ i = 0, 1, \ldots, \log_{p}(\exp(G/G')) \}$. We now investigate the set $\cod_{\nl}(G)$ for the cases (i)--(iv). We need to find out the orders of those subgroups $N$ of $Z(G)$ that do not contain $G'$, and that their quotients in $Z(G)$ are cyclic. \\
	First, suppose that $Z(G)$ is cyclic. Since $G'$ is elementary abelian and is contained in $Z(G)$, we must have $G' \cong C_{p}$. If $N$ is a nontrivial subgroup of $Z(G)$, then $N$ contains $G'$. Such an $N$ is not a suitable candidate for $\ker(\chi)$. Hence, $N$ must be the trivial subgroup of $Z(G)$, that is, $\ker(\chi) = 1$ for all $\chi \in \nl(G)$. Thus, $\cod(\chi) = |G/ \ker(\chi)|/\chi(1) = |G|/|G/Z(G)|^{1/2} = (|G||Z(G)|)^{1/2}$, for all $\chi \in \nl(G)$. Finally, $\cod(G) = \cod_{\lin}(G) \cup \cod_{\nl}(G) =  \{ p^i, (|G||Z(G)|)^{1/2}~:~ i = 0, 1, \ldots, \log_{p}(\exp(G/G')) \}$.\\
	Next, suppose that $Z(G)$ is elementary abelian, of rank $r$. Now, the required subgroup $N$ of $Z(G)$ must satisfy $Z(G)/N \cong C_p$ which implies that $|N| = |Z(G)|/p = p^{r-1}$. Thus, for $N$, choose a subgroup of $Z(G)$ of rank $r-1$ that does not contain $G'$. Note that there are no other possibilities for the order of such a subgroup $N$. Then, for all $\chi\in \nl(G)$, $|\ker(\chi)| = |N| = |Z(G)|/p$, and thus, $\cod(\chi) = p\cdot |G/Z(G)|^{1/2}$. Finally, $\cod(G) = \cod_{\lin}(G) \cup \cod_{\nl}(G) =  \{ p^i, p\cdot |G/Z(G)|^{1/2} ~:~ i = 0, 1, \ldots, \log_{p}(\exp(G/G')) \}$.\\
	Next, suppose that $d(G') = d(Z(G))$ and $Z(G) = \langle a_1,a_2,\ldots,a_r \rangle \cong C_{p^{l_1}} \times C_{p^{l_2}} \times \cdots \times C_{p^{l_r}}$ where $l_1\geq l_2\geq \cdots \geq l_r$. Now, for the required subgroup $N$, $Z(G)/N$ must be cyclic, thus $|Z(G)/N| \leq p^{l_1}$ which implies that $|N|\geq p^{l_2 + \cdots + l_r}$. Since $G'$ must not be contained in $N$, we must have $d(N)<r$. Then $d(N)$ must be $r-1$ (to ensure that $Z(G)/N$ is cyclic). This forces that  
	\[ |N| = p^{\sum_{k=1, k\neq j}^{r} l_k}, ~ \text{ for some } 1\leq j \leq r. \] 
	Conversely, if, for each $1\leq j \leq r$, we take $N = \langle a_1,\ldots,a_{j-1}, a_{j+1}, \ldots, a_r \rangle$, then $|N| = p^{\sum_{k=1, k\neq j}^{r} l_k}$, for each such $j$. Thus, if $\chi \in \nl(G)$, then $|\ker(\chi)|$ belongs to the set 
	\[ \{ p^{\sum_{k=1, k\neq j}^{r} l_k}~:~ j =1,2,\ldots,r \}. \] Hence, $\cod(\chi) = |G/\ker(\chi)|/\chi(1) = \frac{|G/Z(G)||Z(G)/\ker(\chi)|}{|G/Z(G)|^{1/2}}$.
	Thus, \[ \cod_{\nl}(G) = \{ p^{l_j}\cdot|G/Z(G)|^{1/2}~:~ j =1,2,\ldots,r \}. \] Finally, we get $\cod(G) = \cod_{\lin}(G) \cup \cod_{\nl}(G) =  \{ p^i, p^{l_j}\cdot|G/Z(G)|^{1/2}~:~ i = 0, 1, \ldots, \log_{p}(\exp(G/G')), j = 1,2,\ldots, r \} $. \\
	Finally, suppose $d(G') < d(Z(G)) = 2$. Since $G$ is a VZ $p$-group, $G'$ must be cyclic of order $p$. We have two cases here. \\
	\noindent {\bf Case I:} We can choose generators of $Z(G)$ such that $Z(G) = \langle \alpha, \beta \rangle \cong C_{p^{l_1}} \times C_{p^{l_2}}  ~ (l_1\geq l_2)$ and $G' = \langle \beta^{p^{l_2 - 1}} \rangle \cong C_p$. Now, for the required subgroup $N$, $Z(G)/N$ must be cyclic, thus $|Z(G)/N| \leq p^{l_1}$ which implies that $|N|\geq p^{l_2}$. For each $k$, where $l_2 \leq k \leq l_1$, consider the subgroups $N_k = \left\langle \alpha^{p^{l_1 - k}} \beta \right\rangle \cong C_{p^{k}}$. Then for each $l_2 \leq k \leq l_1$, $G'$ is not contained in $N_k$, and $Z(G)/N_k \cong C_{p^{l_1+l_2-k}}$. Thus, order of the kernel of a nonlinear irreducible character $\chi$ has the following possibilities:
	\[ p^k \text{ for } k = l_2, l_2 + 1, \ldots, l_1. \]
	Then $\cod(\chi) = |G/Z(G)|^{1/2} \cdot |Z(G)/N|$ has the possibilities $|G/Z(G)|^{1/2} \cdot p^{l_1+l_2-k}$ for $k = l_2, l_2 + 1, \ldots, l_1$. Thus, 
	\[ \cod_{\nl}(G) = \{ p^j \cdot |G/Z(G)|^{1/2} ~:~ j = l_2, l_2+1,\ldots, l_1 \}. \]
	\noindent {\bf Case II:} We can choose generators of $Z(G)$ such that $Z(G) = \langle \alpha, \beta \rangle \cong C_{p^{l_1}} \times C_{p^{l_2}}  ~ (l_1\geq l_2)$ and $G' = \langle \alpha^{p^{l_1 - 1}} \rangle \cong C_p$. Again, for the required subgroup $N$ (here $d(Z(G))=2$), $N$ and $Z(G)/N$ must be cyclic, thus $|Z(G)/N| \leq p^{l_1}$ which implies that $|N|\geq p^{l_2}$. Let $g$ be an element of $Z(G)$ of order greater than $p^{l_2}$. Then $g = \alpha^i \beta^j$ where $i$ is chosen so that the order of $\alpha^i$ is greater than $p^{l_2}$. Then $\alpha^{p^{l_1 - 1}} \in \langle \alpha^i \beta^j \rangle$, and hence, $G'$ is contained in any cyclic subgroup of order greater than $p^{l_2}$. The only possibility for the order of the required subgroup $N$ is $p^{l_2}$. Consider $N = \langle \beta \rangle \cong C_{p^{l_2}}$. Then $N$ and $Z(G)/N$ are cyclic, and $G'$ is not contained in $N$. Hence, $|\ker(\chi)|$ must be $p^{l_2}$ for all $\chi\in \nl(G)$. Hence, $\cod(\chi) = |G/Z(G)|^{1/2} \cdot p^{l_1+l_2-l_2} = p^{l_1} \cdot |G/Z(G)|^{1/2} = \exp(Z(G))\cdot|G/Z(G)|^{1/2}$, for all $\chi\in \nl(G)$. Thus, in this case, we have \[ \cod_{\nl}(G) = \{ \exp(Z(G))\cdot|G/Z(G)|^{1/2} \}.  \]
	Since $\cod(G) = \cod_{\lin}(G) \cup \cod_{\nl}(G)$ where $\cod_{\lin}(G) = \{ p^i~:~ i = 0,1,\ldots,\log_{p}(\exp(G/G')) \}$, we obtain the desired result.
	This concludes the proof. \qed \\
	
	With the help of Theorem \ref{thm: VZ p-group}, we obtain $\cod(G)$ when $G$ is an extraspecial $p$-group. Note that a $p$-group $G$ is called extraspecial if $Z(G) = G'$ and the quotient $G/Z(G)$ is a non-trivial elementary abelian $p$-group. It is easy to see that every extraspecial $p$-group is a VZ $p$-group. We have the following corollary.
	
	\begin{corollary} \label{cor:extraspecial p-group}
		Let $G$ be an extraspecial $p$-group. Then $\cod(G) = \{ 1, p, (p\cdot |G|)^{1/2} \}$.
	\end{corollary}
	\begin{proof}
		Since $G$ is VZ as well, $Z(G) = G'$ is cyclic of order $p$ and $G/G'$ is elementary abelian. By Theorem \ref{thm: VZ p-group}, $\cod(G) =  \{ p^i, (|G||Z(G)|)^{1/2}~:~ i = 0, 1, \ldots, \log_{p}(\exp(G/G')) \} = \{ 1, p, (p\cdot |G|)^{1/2} \}$.
	\end{proof}
	
	Now, we deal with the Camina $p$-groups. First, we mention the preliminaries we utilize to obtain the character codegrees set. The pair $(G,N)$ is said to be a Camina pair if for every $g\in G \setminus N,~ gN \subseteq Cl_{G}(g),$ the conjugacy class of $g$. It is easy to see that $Z(G) \leq  N \leq G{}'$. When $N=G'$, the group $G$ is called a Camina group. By \cite[Lemma 3]{M1992}, if $G$ is a Camina group, then for all $g\in G \setminus G{}'$, $\chi(g) = 0$  for all $\chi\in \nl(G)$.
	Dark and Scoppola \cite{DS1996} proved that nilpotency class of $G$ is at most $3$. If $G$ is a Camina $p$-group of class 2, then $G' \leq Z(G)$. Thus, $G{}' = Z(G)$, and hence by \cite[Lemma 3]{M1992} and \cite[Corollary 2.3]{M1981}, $G$ is a VZ $p$-group such that $Z(G)$ and $G/G'$ are elementary abelian. If $G$ is a Camina $p$-group of class 3, then
	$(G, Z(G))$ is a Camina pair, and $|G/G'|=p^{2n}, |G'/Z(G)|=p^{n}$ and $|G/Z(G)|=p^{3n},$ where $n$ is even (see \cite[Theorem 5.2]{M1981}). Further, by \cite[Corollary 2.3 and 5.3]{M1981}, $G/G'$, $G'/Z(G)$ and $Z(G)$ are elementary abelian $p$-groups. Moreover, $\nl(G) = \Irr(G|Z(G)) \sqcup \nl(G/Z(G))$ as a disjoint union and $\cd(G) = \{1, p^{n}, p^{3n/2}\}$ (see \cite[Section 3, pp. 51--52]{PS2014} for the detailed discussion). We obtain the following result for Camina $p$-groups.
	
	\begin{theorem} \label{thm:CodG Camina pgroup}
		Let $G$ be a Camina $p$-group.
		\begin{enumerate}
			\item If the nilpotency class of $G$ is 2, then $\cod(G) = \{ 1, p, p\cdot|G/Z(G)|^{1/2} \}$.
			\item If the nilpotency class of $G$ is 3, then  $\cod(G) = \left\{ 1, p, p \cdot |G/G'|^{1/2}, p \cdot |G/Z(G)|^{1/2}  \right\}$. 
		\end{enumerate} 
	\end{theorem}
	\begin{proof}
		Let $G$ be a Camina $p$-group. If $\cl(G) = 2$, then $G$ is a VZ $p$-group such that $Z(G)$ and $G/G'$ are elementary abelian. By Theorem \ref{thm: VZ p-group}, $\cod(G) = \{ 1, p, p\cdot|G/Z(G)|^{1/2} \}$. Now, suppose $\cl(G) = 3$. Since $G/G'$ is elementary abelian, $\cod_{\lin}(G) = \{ 1, p \}$. Now, we compute $\cod_{\nl}(G)$. Since $\nl(G) = \Irr(G|Z(G)) \sqcup \nl(G/Z(G))$ as a disjoint union, first take $\chi \in \Irr(G|Z(G))$. Since $(G, Z(G))$ is a Camina pair, $\chi(g) = 0$ for all $g\in G\setminus Z(G)$. By \cite{PS2014}, $\chi(1) = p^{3n/2}$, and thus, $\chi = p^{3n/2}\lambda$, for some $\lambda \in \Irr(Z(G))$. Then $\ker(\chi) = \ker(\lambda)$. Since $Z(G)$ is elementary abelian and $Z(G)/\ker(\lambda)$ must be cyclic, we must have $|\ker(\lambda)| = |Z(G)|/p$. Therefore, \[ \cod(\chi) = \frac{|G/\ker(\chi)|}{\chi(1)} = \frac{|G|}{|\ker(\lambda)|\chi(1)} = \frac{|G|\cdot p}{|Z(G)|\cdot |G/Z(G)|^{1/2} } = p\cdot |G/Z(G)|^{1/2}. \]
		Now, suppose $\nu \in \nl(G/Z(G))$. Since $(G,G')$ is a Camina pair, $(G/Z(G),G'/Z(G))$ is also a Camina pair. By \cite[Corollary 5.3]{M1981}, $Z(G/Z(G)) = G'/Z(G) = [G/Z(G), G/Z(G)]$, that is, $G/Z(G)$ is a Camina $p$-group with $\cl(G) = 2$. Thus, $G/Z(G)$ is a VZ $p$-group such that $Z(G/Z(G))$ and $((G/Z(G))/(G'/Z(G)))$ are elementary abelian. Hence, $|\ker(\nu)| = Z(G/Z(G))/p = |G'/Z(G)|/p$. Then \[  \cod(\nu) = \frac{|(G/Z(G))/\ker(\nu)|}{\nu(1)} = \frac{|G/Z(G)|\cdot p}{|G'/Z(G)|\cdot \nu(1)} = \frac{p^{3n}\cdot p}{p^n \cdot p^n} = p^{n+1} = p \cdot |G/G'|^{1/2}.    \]
		Therefore, $\cod(G) = \left\{ 1, p, p \cdot |G/G'|^{1/2}, p \cdot |G/Z(G)|^{1/2}  \right\}$.
	\end{proof}

	\section{Character Codegrees Sets of $p$-Groups of Order $\leq p^5$, where $p$ is an odd prime} \label{sec:groups less than p6}
	
	In this section, we compute the character codegrees set for $p$-groups of order $\leq p^5$, where $p$ is an odd prime.
	The groups of order $\leq p^5~(p>2)$ are listed in \cite{J1980}. There are 2 isoclinic families for groups of order $p^3$, 3 for groups of order $p^4$, and 10 for groups of order $p^5$. In all the cases, the family $\Phi_1$ consists of all abelian	groups. The structure of groups of order $p^4$ and $p^5$ when $p=3$ is different from when $p$ is a prime greater than 3. We theoretically prove the results for $p\geq 5$. 
Since the character codegrees sets of the abelian groups can be computed by Corollary \ref{cor:linearcod}, we focus on the non-abelian groups. We reference extensively the list of presentations, nilpotency classes, character degree sets, and other data available in \cite{J1980}, and use its group identifiers. For completeness, we compute $\cod(G)$ for the groups of order $3^n$ (for $n=3,4,5$) using {\sc Magma}; we summarize the results in Table \ref{t:3-groups}. Henceforth in this section, $p$ denotes a prime greater than 3, unless stated otherwise.

	
	\subsection{Character codegrees sets of $p$-groups of order less than $p^5$} \label{subsec:codGless than p5}
	We prove the following. 
	
	\begin{theorem} \label{thm:nonabelianlessthanp5}
		Let $G$ be a non-abelian $p$-group.
		\begin{enumerate}
			\item If $|G| = p^3$, then $\cod(G) = \{ 1, p, p^2 \}$.
			\item Let $|G| = p^4$ and 
			\begin{enumerate}
				\item [\rmfamily(i)] suppose $\cl(G) = 2$. If $Z(G)$ is cyclic, then $\cod(G) = \{1, p, p^3  \}$, or $\{ 1, p, p^2, p^3 \}$. If $Z(G)$ is non-cyclic, then $\cod(G) = \{ 1, p, p^2 \}$ .
				\item [\rmfamily(ii)] suppose $\cl(G) = 3$. Then $\cod(G) = \{ 1, p, p^2, p^3 \}$.
			\end{enumerate}
		\end{enumerate} 
	\end{theorem}
	\begin{proof}  Let $G$ be a non-abelian group. If $|G| = p^3$, then $G$ is an extraspecial group. 	By Corollary \ref{cor:extraspecial p-group},	 $\cod(G) = \{ 1, p, \left( p \cdot p^3 \right)^{1/2}  \} = \{ 1, p, p^2 \}$. 
		Now, suppose $|G| = p^4$. If $\cl(G) = 2$, then $G$ belongs to the isoclinic family $\Phi_2$, and hence, $|Z(G)| = p^2$, $|G'| = p$ and $\cd(G) = \{ 1, p \} = \{ 1, |G:Z(G)|^{1/2} \}$ (see \cite[Table 4.1]{J1980}). Then $G$ is a VZ-group \cite[Theorem A]{FAM2001}. If $Z(G)\cong C_p \times C_p$, by Theorem \ref{thm: VZ p-group}, $\cod(G) = \cod_{\lin}(G) \cup \{ p\cdot |G/Z(G)|^{1/2} \} = \cod_{\lin}(G) \cup \{ p^2 \}$. By Corollary \ref{cor:linearcod}, $\cod_{\lin}(G) = \{ p^i~:~ i=0,1,\ldots, \log_{p}(\exp(G/G')) \}$. Hence, we need $\exp(G/G')$.
		Since $Z(G)$ and $G/Z(G)$ are elementary abelian, $\exp(G)$ divides $p^2$. Then $\exp(G/G')$ is either $p$ or $p^2$. In both the cases, we get $\cod(G) = \{ 1, p, p^2 \}$.\\
		When $Z(G)$ is cyclic, then $G = \phi_2(31)$, or $\phi_2(211)b$. Since $G$ is a VZ $p$-group,
		by Theorem \ref{thm: VZ p-group}, $\cod(G) = \cod_{\lin}(G) \cup \{ p\cdot \left(|G||Z(G)|^{1/2}\right) \} = \cod_{\lin}(G) \cup \{ p^3 \}$.
		We compute the exponent of $G/G'$ for each group separately. If  
		\[ G = \phi_2(31) = \langle \alpha, \alpha_1, \alpha_2 ~|~ [\alpha_1, \alpha] = \alpha^{p^2} = \alpha_2, \alpha_{1}^p = \alpha_{2}^{p} = 1 \rangle, \]
		then $G' = \langle \alpha^{p^2} \rangle \cong C_p$ and $G/G' = \langle \alpha G', \alpha_{1}G' \rangle \cong C_{p^2} \times C_{p}$. Hence, $\exp(G/G') = p^2$ and thus, $\cod(G) = \{ 1, p, p^2, p^3 \}$. Lastly, if
		\[ G = \phi_2(211)b = \langle \alpha, \alpha_1, \alpha_2, \gamma ~|~ [\alpha_1, \alpha] = \gamma^{p} = \alpha_2, \alpha^p = \alpha_{1}^p = \alpha_{2}^{p} = 1 \rangle, \]
		then $G' = \langle \gamma^{p} \rangle \cong C_p$ and $G/G' = \langle \alpha G', \alpha_{1}G', \gamma G' \rangle \cong C_{p} \times C_{p} \times C_{p}$. Then $\exp(G/G') = p$, and thus, $\cod(G) = \{ 1, p, p^3 \}$.\\
		Finally, if $\cl(G) = 3$, then $G$ is a maximal class $p$-group. By \cite[Theorem A]{CL2020}, $\cod(G) = \{ 1, p, p^2, p^3 \}$.
		This completes the proof.
	\end{proof}
	
	From Theorem \ref{thm:nonabelianlessthanp5}, for any non-abelian group of order $p^3$, $\cod(G) = \{1, p, p^2\}$. However, we obtain different possibilities for $\cod(G)$ when $G$ is a non-abelian group of order $p^4$. To avoid confusion, we list down the character codegrees set corresponding to each group of order $p^4$ in Table \ref{t:p34}.
	
	\begin{longtable}[c]{| l| c |}
		\caption{Character codegrees set of non-abelian groups of order $p^4$ \label{t:p34}}\\

		\hline
		Group $G$ &  $\cod(G)$  \\ 
		\hline
		\endfirsthead
		
		\hline
		\multicolumn{2}{|c|}{Continuation of Table \ref{t:p34}}\\
		\hline
		Group $G$ &  $\cod(G)$ \\  
		\hline
		\endhead
		
		\hline
		\endfoot

		\endlastfoot
		\hline
		\vtop{\hbox{\strut $\phi_{2}(211)a$, } \hbox{\strut $\phi_{2}(1^4)$, }  \hbox{\strut $\phi_{2}(22)$, } \hbox{\strut $\phi_{2}(211)c$ }}  & $\{ 1, p, p^2 \}$  \\
		\hline
		\vtop{\hbox{\strut $\phi_{2}(31)$ } }  & $\{ 1, p, p^2, p^3 \}$ \\
		\hline
		$\phi_{2}(211)b$ &  $\{ 1, p, p^3 \}$ \\
		\hline
		\vtop{\hbox{\strut $\phi_{3}(211)a$, } \hbox{\strut $\phi_{3}(1^4)$, } \hbox{\strut $\phi_{3}(211)b_r$ (for $r=1, \nu$) } } &  $\{ 1, p, p^2, p^3 \}$ \\
		\hline
		

	\end{longtable}

	
	\subsection{Character codegrees sets of $p$-groups of order $p^5$} \label{subsec:codG p5}
	In this subsection, we deal with the groups of order $p^5$. 
		The groups of order $p^5$ are classified into 10 isoclinism
			families, $\Phi_{i}$ for $1 \leq i \leq 10$, where $\Phi_1$ consists of all abelian
			groups (see \cite{J1980}). The groups with nilpotency class 2 belong to the family $\Phi_i$ for $i=2,4,5$, those of nilpotency class 3 belong to the family $\Phi_{j}$ for $j=3,6,7,8$, and the maximal class groups belong to the family $\Phi_{k}$ for $k = 9,10$. For $G\in \Phi_{i}$, $i=2,3,4,6, 9$, $\cd(G) = \{1, p\}$; for $G\in \Phi_5$, $\cd(G) = \{ 1, p^2 \}$, and for $G\in \Phi_{j}$, $j=7,8,10$, $\cd(G) = \{ 1, p, p^2 \}$. Note that for $G\in \Phi_2 \cup \Phi_5$, $\cd(G) = \{ 1, |G/Z(G)|^{1/2} \}$, that is, $G$ is a VZ $p$-group. We utilize Theorem \ref{thm: VZ p-group} and Corollary \ref{cor:extraspecial p-group} to obtain $\cod(G)$ when $G\in \Phi_2 \cup \Phi_5$. 
	
	
	\begin{lemma} \label{lemma:groupsp5 phi2}
		Let  $G$ be a group of order $p^5$ belonging to the family $\Phi_{2}$ such that 
		\begin{enumerate}
			\item [(i)] $Z(G)\cong C_{p^3}$. Then $\cod(G) = \{ p^i, p^4 ~:~ i = 0, 1, \ldots, \log_{p}(\exp(G/G')) \}$.
			\item [(ii)] $Z(G)\cong C_{p^2}\times C_p$. Then $\cod(G) = \{ p^i, p^2, p^3 ~:~ i = 0, 1, \ldots, \log_{p}(\exp(G/G')) \}$, or $\{ p^i, p^3 ~:~ i = 0, 1, \ldots, \log_{p}(\exp(G/G')) \}$.
			\item [(iii)] $Z(G)\cong C_{p}^3$. Then $\cod(G) = \{ p^i, p^2 ~:~ i = 0, 1, \ldots, \log_{p}(\exp(G/G')) \}$.
		\end{enumerate}
		
	\end{lemma}
	\begin{proof}
		Let  $G$ be a group of order $p^5$ belonging to the family $\Phi_{2}$. Then $|Z(G)| = p^3$ and $\cd(G) = \{ 1, p \} = \{ 1, |G/Z(G)|^{1/2} \}$. Hence, $G$ is a VZ $p$-group where $Z(G) \cong C_{p^3}$, $C_{p^2}\times C_p$, or $C_{p}^3$. Then, the result follows from Theorem \ref{thm: VZ p-group}.
	\end{proof}

	\begin{remark} \label{remark: VZ p-groups phi2}
		\textnormal{Here, we mention the information required to compute the exact $\cod(G)$ when $G$ is a group of order $p^5$ belonging to $\Phi_2$. We segregate the groups based on their center types here.\\
			Case I - $Z(G) \cong C_{p^3}$: In this case, $G \in \{ \phi_{2}(41), \phi_{2}(311)b \}$.\\
			Case II - $Z(G) \cong C_{p^2}\times C_p$: Take
			\begin{align*}
				X_1 = & \{  \phi_{2}(221)b, \phi_{2}(2111)d,  \phi_{2}(32)a_2, \phi_{2}(311)c \} \text{ and }\\
				X_2 = & \{ \phi_{2}(31),
				\phi_{2}(2111)b, \phi_{2}(32)a_1, \phi_{2}(221)c \}.
			\end{align*}
			Then $G \in X_1 \cup X_2$, when $Z(G) \cong C_{p^2}\times C_p$. For all $G \in X_1$, we can choose $a, b \in G$ such that $Z(G) = \langle a, b \rangle\cong C_{p^2} \times C_{p}$ and $G' = \langle b \rangle \cong C_{p}$, whereas, 
			for all $G \in X_2$, we can choose $a,b \in G$ such that $Z(G) = \langle a,b \rangle\cong C_{p^2} \times C_{p}$ and $G' = \langle a^p \rangle \cong C_{p}$.\\
			Case III - $Z(G) \cong C_{p}^3$: In this case, $G \in \{ \phi_{2}(221)a, \phi_{2}(2111)a, \phi_{2}(2111)c, \phi_{2}(1^5), \phi_{2}(221)d \}$.\\
			By Lemma \ref{lemma:groupsp5 phi2}, we just need the exponent of $G/G'$ for each $G\in \Phi_{2}$ to compute $\cod(G)$. For each group $G \in \Phi_2$, we display the corresponding $\exp(G/G')$ and $\cod(G)$ in Table \ref{t:VZ groups phi2 phi5}.   }
	\end{remark}
	
	\begin{lemma} \label{lemma:p5 phi5}
		Let $G$ be a non-abelian group of order $p^5$ belonging to the family $\phi_5$. Then $\cod(G) = \{ 1, p, p^3 \}$.
	\end{lemma}
	\begin{proof}
		From \cite[Section 4.5]{J1980}, if $G\in \Phi_5$, then $G' = Z(G)\cong C_p$ and $\cd(G) = \{  1, p^2 \}$. Thus, $\cd(G) = \{ 1, |G/Z(G)|^{1/2}\}$, that is, $G$ is a VZ $p$-group. Since $Z(G) = G'$, $G$ is also an extraspecial $p$-group. By Corollary \ref{cor:extraspecial p-group}, $\cod(G) = \{ 1, p, p^3 \}$
	\end{proof}
	
	We present the character codegrees set for each VZ $p$-group of order $p^5$ in Table \ref{t:VZ groups phi2 phi5}.
	
	\begin{longtable}[c]{|c|c|c||c|c|c|}
		\caption{Character codegrees set of VZ $p$-groups of order $p^5$  \label{t:VZ groups phi2 phi5}}\\
		
		\hline
		Group $G$ & $\exp(G/G')$ & $\cod(G)$ & 	Group $G$ & $\exp(G/G')$ & $\cod(G)$ \\ 
		\hline
		\endfirsthead
		
		\hline
		\multicolumn{6}{|c|}{Continuation of Table \ref{t:VZ groups phi2 phi5}}\\
		\hline
		Group $G$ & $\exp(G/G')$ & $\cod(G)$ & 	Group $G$ & $\exp(G/G')$ & $\cod(G)$ \\  
		\hline
		
		\endhead
		
		\hline
		\endfoot

		\endlastfoot
		\multicolumn{6}{| c |}{$Z(G) \cong C_{p^3}$}\\
		\hline
		\vtop{\hbox{\strut $\phi_{2}(41)$ }} & $p^3$ & $\{ 1, p, p^2, p^3, p^4 \}$ & \vtop{\hbox{\strut $\phi_{2}(311)b$ }} & $p^2$ & $\{ 1, p, p^2, p^4 \}$ \\
		\hline
		\multicolumn{6}{| c |}{$Z(G) \cong C_{p^2}\times C_{p}$}\\
		\hline
		\vtop{\hbox{\strut $\phi_{2}(221)b$, } \hbox{\strut $\phi_{2}(2111)d$ }} & $p^2$ & $\{ 1, p, p^2, p^3 \}$ & \vtop{\hbox{\strut $\phi_{2}(32)a_2$, } \hbox{\strut $\phi_{2}(311)c$ }} & $p^3$ & $\{ 1, p, p^2, p^3 \}$\\
		\hline
		\vtop{\hbox{\strut $\phi_{2}(311)a$, } \hbox{\strut $\phi_{2}(32)a_1$, } \hbox{\strut $\phi_{2}(221)c$ }} & $p^2$ & $\{ 1, p, p^2, p^3 \}$ & \vtop{\hbox{\strut $\phi_{2}(2111)b$ }} & $p$ & $\{ 1, p, p^3 \}$\\
		\hline
		\multicolumn{6}{| c |}{$Z(G) \cong C_{p}^{3}$}\\
		\hline
		\vtop{\hbox{\strut $\phi_{2}(221)a$, } \hbox{\strut $\phi_{2}(2111)c$, } \hbox{\strut $\phi_{2}(221)d$ }} & $p^2$ & $\{ 1, p, p^2 \}$ & \vtop{\hbox{\strut $\phi_{2}(2111)a$, } \hbox{\strut $\phi_{2}(1^5)$ }} & $p$ & $\{ 1, p, p^2 \}$\\
		\hline
		\multicolumn{6}{| c |}{$Z(G) \cong C_{p}$}\\
		\hline
		\vtop{\hbox{\strut $\phi_{5}(2111)$, } \hbox{\strut $\phi_{5}(1^5)$ }} & $p$ & $\{ 1, p, p^3 \}$ & - & - & - \\
		\hline
	\end{longtable}

	\begin{lemma} \label{lemma:p5 cl3 cdG1,p}
		Let $G$ be a non-abelian group of order $p^5$ such that $\cl(G) \geq 3$ and $\cd(G) = \{ 1, p \}$. 
		\begin{enumerate}
			\item If $Z(G)$ is non-cyclic, then $\cod(G) = \{ 1, p, p^2, p^3 \}$.
			\item If $Z(G)$ is cyclic and
			\begin{enumerate}
				\item $\cl(G) = 3$, then $|\cod(G)|=4$ and $1, p, p^4 \in \cod(G)$. Further, if $G/G'$ is not elementary abelian, then $\cod(G) = \{ 1, p, p^2, p^4 \}$.
				\item $\cl(G) = 4$, then $\cod(G) = \{ 1, p, p^2, p^3, p^4 \}$.
			\end{enumerate} 
		\end{enumerate}
	\end{lemma}
	\begin{proof}
		Since $G$ is nontrivial, $1,p \in \cod(G)$. Also, $\cd(G) = \{1, p\}$ and $\cl(G) \geq 3$ implies that $G\in \Phi_{3}\cup\Phi_6\cup\Phi_9$ (see \cite{J1980}). For all such $G$, $\exp(G) \leq p^3$, which implies that $\exp(G/G') \leq p^3$. Thus, by Corollary \ref{cor:linearcod}, $\cod(\eta) \leq p^3$, for all $\eta\in \lin(G)$. Now, let $\chi$ be a nonlinear irreducible character of $G$. If   $|G/\ker(\chi)| \leq p^2$, then $G/\ker(\chi)$ is abelian and hence $G' \subseteq \ker(\chi)$, which is not possible. Hence, $|G/\ker(\chi)|\geq p^3$, and thus, $|\ker(\chi)| \leq p^2$. \\
		First, assume that $Z(G)$ is non-cyclic. Since $\chi$ cannot not be faithful, $p \leq |\ker(\chi)| \leq p^2$, which implies that $p^2 \leq \cod(\chi) \leq p^3$. By \cite[Theorem 1.2]{DL2016}, $\cl(G) \geq 3$ implies that $|\cod(G)| > 3$.  Thus, we must have $\cod(G) = \{ 1, p, p^2, p^3 \}$ (since 1 and $p$ are already in $\cod(G)$).\\
		Now, suppose $Z(G)$ is cyclic. If $\cl(G) = 4$, then $G$ is a maximal class $p$-group with $\cd(G) = \{ 1, p \}$. By \cite[Theorem A]{CL2020 maximal}, $\cod(G) = \{ 1, p, p^2, p^3, p^4 \}$.
		Now, suppose $\cl(G) = 3$. Since $Z(G)$ is cyclic, there exists a faithful irreducible character, say $\delta$, of $G$. Then $\cod(\delta) = p^4$. Again, since $1,p \in \cod(G)$, we must have $|\cod(G)| > 3$ (from \cite[Theorem 1.2]{DL2016}). 
		By \cite[Theorem A]{CL2020 maximal}, $\cod(G) \neq \{ 1,p,p^2,p^3,p^4 \}$, that is, $|\cod(G)|<5$. Hence, $|\cod(G)| = 4$. Further, if $G/G'$ is not elementary abelian, then from \cite[Corollary 2.5]{DL2016}, $p^2\in \cod(G)$. Therefore, $\cod(G) = \{ 1, p, p^2, p^4 \}$. This completes the proof.
	\end{proof}
	
	\begin{lemma} \label{lemma:p5 cl3 cdG1,p,p^2}
		Let $G$ be a non-abelian group of order $p^5$ such that $\cl(G) \geq 3$ and $\cd(G) = \{ 1, p, p^2 \}$. Then $\cod(G) = \{ 1, p, p^2, p^3 \}$.
	\end{lemma}
	\begin{proof}
		Since $G$ is nontrivial, $1, p\in \cod(G)$. Also, $\cd(G) = \{1, p, p^2\}$ and $\cl(G) \geq 3$ implies that $G\in \Phi_{7}\cup\Phi_8\cup\Phi_{10}$ (see \cite{J1980}). For all such $G$, $\cl(G) \geq 3$ implies that $\exp(G/G') \leq p^3$. Thus, by Corollary \ref{cor:linearcod}, $\cod(\eta) \leq p^3$, for all $\eta\in \lin(G)$.
		Let $\chi$ be a nonlinear irreducible character of degree $p$ in $G$. Then, $|G/\ker(\chi)|\geq p^3$, and thus, $|\ker(\chi)| \leq p^2$. If $\chi$ is not faithful, then $|\ker(\chi)|= p$, or $p^2$. If $\chi(1) = p$, then $\cod(\chi) = p^3$, or $p^2$. Otherwise, if $\chi(1) = p^2$, then $\cod(\chi) = p^2$, or $p$. On the other hand, suppose $\chi$ is faithful. By \cite[Lemma 22]{PU2023}, $\chi(1) = p^2$. Then $\cod(\chi) = |G|/\chi(1) = p^3$. Thus, $\cod(\chi) \leq p^3$ for all $\chi \in \nl(G)$. Since $1, p\in \cod(G)$ and $\cl(G) \geq 3$, we get $|\cod(G)| = 4$ by \cite[Theorem 1.2]{DL2016}.
		Therefore, $\cod(G) = \{ 1, p, p^2, p^3 \}$.
	\end{proof}
	
To deal with the groups with nilpotency class greater than 2, we obtain Lemma \ref{lemma:p5 cl3 cdG1,p} and \ref{lemma:p5 cl3 cdG1,p,p^2}.

	\begin{lemma} \label{lemma:groupsp5 phi3}
		Let  $G$ be a group of order $p^5$ belonging to the family $\Phi_{3}$. Then $\cod(G) = \{ 1, p, p^2, p^3 \}$, or $\{ 1, p, p^2, p^4 \}$.
	\end{lemma}
	\begin{proof}
		For $G\in \Phi_{3}$, $\cl(G) = 3$, $Z(G) \cong C_{p^2}$, or $C_{p}\times C_{p}$, and $\cd(G) = \{ 1, p \}$ (see \cite{J1980}). If the center is non-cyclic, then $\cod(G) = \{ 1, p, p^2, p^3 \}$ from Lemma \ref{lemma:p5 cl3 cdG1,p}. If the center is cyclic, then we have the following groups: $\phi_{3}(311)a$, $\phi_{3}(311)b_r$ $(r=1,\nu)$, and $\phi_3(2111)c$.
		It is easy to check that if $G = \phi_{3}(311)a$, or $\phi_{3}(311)b_r$ $(r=1,\nu)$, then $G/G'$ is not elementary abelian. By Lemma \ref{lemma:p5 cl3 cdG1,p}, we get that $\cod(G) = \{ 1,p,p^2, p^4 \}$.\\ Now, we are left with 
		\[ G = \phi_3(2111)c = \langle \alpha, \alpha_1, \alpha_2, \alpha_3, \gamma ~|~ [\alpha_1, \alpha] = \alpha_2, [\alpha_{2}, \alpha] = \gamma^{p} = \alpha_3, \alpha^{p} = \alpha_{i}^p  = 1 ~(i =1,2,3) \rangle. \]
		Here $G' = \langle \alpha_{2}, \gamma^p \rangle \cong C_p \times C_p$, $G/G' \cong C_{p}^3$ and $Z(G) = \langle \gamma \rangle \cong C_{p^2}$. From Corollary \ref{cor:linearcod}, $\cod(\eta) = p$ if $1 \neq \eta \in \lin(G)$. We now compute $\cod_{\nl}(G)$.
		Let $\chi \in \nl(G)$. If $\chi$ is faithful (existence of such an irreducible character is ensured by the cyclic center), then $\cod(\chi) = p^4$. Now, assume that $\chi$ is not faithful. Take $A = \langle \alpha_{1}, \alpha_{2}, \gamma \rangle \cong C_p \times C_{p} \times C_{p^2}$. Then $A$ is a maximal abelian subgroup of $G$ containing $G'$. By \cite[Theorem 2]{BKP2013}, $\chi = \lambda\ind_{A}^{G}$ for some $\lambda\in \lin(A)$ with $\ker(\lambda) = D$ where $D\leq A$ such that $A/D$ is cyclic. Note that $\ker(\chi) = \Core_{G}(\ker(\lambda))$. 
		If we take $ D =  \langle \alpha_{1}, \gamma \rangle$, then $A/D$ is cyclic. Here, $\Core_{G}(D) = \langle \gamma \rangle$. Thus, if we take $D = \ker(\lambda)$, then  $\ker(\chi) = \Core_{G}(\ker(\lambda)) = \langle \gamma \rangle$. Hence, $|\ker(\chi)| = p^2$ and thus, $\cod(G) = p^2$. Now, we prove that for any $\chi\in \nl(G)$, $\cod(\chi) \neq p^3$, that is, $|\ker(\chi)| \neq p$. Since $\exp(A) = p^2$ and $A/\ker(\lambda)$ is cyclic, we get $|A/\ker(\lambda)|\leq p^2$. Hence, $|\ker(\lambda)| \geq p^2$. On the other hand, $|\ker(\lambda)| < p^4$, otherwise $\ker(\lambda) = A$ and hence $G' \subset \Core_{G}(\ker(\lambda))$ which is a contradiction. Therefore, $|\ker(\lambda)| = p^2$, or $p^3$. Now, if $|D| = p^2$, then $A/D \cong C_{p^2}$. By \cite[Lemma 4.2]{O'BPU2024}, $\langle \gamma \rangle \cap D = 1$, that is, $Z(G) \cap D = 1$. This implies that $Z(G) \cap \Core_{G}(D) = 1$, and hence, $\Core_{G}(D) = 1$. Hence, $\chi$ is faithful, which is a contradiction. Thus, if $\chi$ is not faithful, then $|\ker(\lambda)|$ must be $p^3$. \\
		Suppose $|D| = p^3$. If $D\cong C_{p^2} \times C_{p}$, then $D \in \{ T_{ij}, T_k ~:~ 1\leq i,j,k \leq p \}$, where $T_{ij} = \langle \gamma\alpha_{1}^i, \alpha_{1}^j\alpha_{2}  \rangle \cong C_{p^2} \times C_p$ and $T_k = \langle \gamma\alpha_{2}^k, \alpha_{1} \rangle \cong C_{p^2} \times C_p$ for $1\leq i,j,k \leq p$. Through routine computation, we obtain that $T_{ij}$ is normal in $G$ for $i=p, 1\leq j \leq p$, and $\Core_{G}(T_{ij}) \cong C_{p^2}$ when $1\leq i<p, 1\leq j \leq p$. Similarly, $\Core_{G}(T_{k}) \cong C_{p^2}$ for all $1\leq k \leq p$.
		On the other hand, if $D \cong C_p^3$, then $D = \langle \alpha_1, \alpha_{2}, \gamma^p \rangle$, and hence $G' \subset D$. This implies that $G' \subset \Core_{G}(D)$, and thus, $D$ is not a suitable choice for $\ker(\lambda)$.
		 Hence, when $|D| = p^3$, then $D $ has to belong to the set $ \{ T_{ij}, T_k ~:~ 1\leq i <p, 1\leq j,k \leq p \}$, to be a candidate for $\ker(\lambda)$. Therefore, $\ker(\chi) = \Core_{G}(\ker(\lambda)) \cong C_{p^2}$, and hence, $\cod(\chi) = p^2$. Therefore, $\cod(G) = \{ 1, p, p^2, p^4 \}$.
	\end{proof}
	
	\begin{lemma} \label{lemma:groupsp5 phi4}
		Let  $G$ be a group of order $p^5$ belonging to the family $\Phi_{4}$. Then $\cod(G) = \{ 1, p, p^2 \}$, $\{ 1, p, p^3 \}$ or $\{ 1, p, p^2, p^3 \}$.
	\end{lemma}
	\begin{proof}
		For $G\in \Phi_{4}$, $\cl(G) = 2$, $G' = Z(G) \cong C_{p}\times C_{p}$, $G/G' \cong C_{p}^3$, and $\cd(G) = \{ 1, p \}$ (see \cite{J1980}). From Corollary \ref{cor:linearcod}, $\cod(\eta) = p$ if $1 \neq \eta \in \lin(G)$. We now compute $\cod_{\nl}(G)$ for each group $G\in \Phi_4$. We deal with each group in the family separately. 
		We start with
		\[ G = \phi_4(221)a = \langle \alpha, \alpha_1, \alpha_2, \beta_1, \beta_2 ~|~ [\alpha_1, \alpha] = \beta_1 = \alpha_{1}^p, [\alpha_{2}, \alpha] =  \beta_2 = \alpha^p, \alpha_{2}^p = \beta_1^{p} = \beta_2^{p} = 1 \rangle. \]
		Here $G' = Z(G) = \langle \alpha^p, \alpha_{1}^p \rangle \cong C_{p}\times C_{p}$. Take $A = \langle \alpha^p, \alpha_{1}, \alpha_{2} \rangle \cong C_p \times C_{p^2} \times C_p$. Then $A$ is a maximal abelian subgroup of $G$ containing $G'$.  Let $\chi \in \nl(G)$. By \cite[Theorem 2]{BKP2013}, $\chi = \lambda\ind_{A}^{G}$ for some $\lambda\in \lin(A)$ with $\ker(\lambda) = D$ where $D\leq A$ such that $A/D$ is cyclic. 
		It is easy to see that, just as in the case of $\phi_{3}(2111)c$ in the proof of Lemma \ref{lemma:groupsp5 phi3}, $|\ker(\lambda)| = p^2$, or $p^3$.
		If $D = \langle \alpha_{1}, \alpha_{2} \rangle$, then $A/D$ is cyclic. In this case, $\Core_{G}(D) = \langle \alpha_1 \rangle$. Hence, if we take $D = \ker(\lambda)$, then  $|\ker(\chi)| =  |\Core_{G}(\ker(\lambda))| = p^2$ which implies that $\cod(\chi) = p^2$. Now, if $D = \langle \alpha^p, \alpha_{2} \rangle$, then $A/D$ is cyclic. Here, $\Core_{G}(D) = \langle \alpha^p \rangle$. Hence, if we take $D = \ker(\lambda)$, then $|\ker(\chi)| = p$ which implies that $\cod(\chi) = p^3$. Thus, $\cod_{\nl}(G) = \{ p^2, p^3 \}$, and finally, $\cod(G) = \{ 1, p, p^2, p^3 \}$. \\
		Using similar arguments, we conclude that $\cod(G) = \{1, p, p^2, p^3 \}$ if $ G \in \{ \phi_4(221)b, \phi_4(221)d_r ~ (r=1,2,\ldots, (p-1)/2), \phi_4(221)e, \phi_4(2111)b, \phi_4(2111)c \}$.\\
		Now, suppose
		\[ G = \phi_4(221)c = \langle \alpha, \alpha_1, \alpha_2, \beta_1, \beta_2 ~|~ [\alpha_1, \alpha] = \beta_1 = \alpha_{1}^p, [\alpha_{2}, \alpha] =  \beta_2 = \alpha_{2}^p, \alpha^p = \beta_1^{p} = \beta_2^{p} = 1 \rangle. \]
		Here $G' = Z(G) = \langle \alpha_1^p, \alpha_{2}^p \rangle \cong C_{p}\times C_{p}$. Take $B = \langle \alpha_1, \alpha_{2} \rangle \cong C_{p^2} \times C_{p^2}$. Then $B$ is a maximal abelian subgroup of $G$ containing $G'$. Let $\nu \in \nl(G)$. Again, by \cite[Theorem 2]{BKP2013}, $\nu = \theta\ind_{B}^{G}$ for some $\theta\in \lin(B)$ with $\ker(\theta) = E$ where $E\leq B$ such that $B/E$ is cyclic. If $E = \langle \alpha_1 \rangle$, then $B/E$ is cyclic and $\Core_{G}(E) = E$. If we take $E = \ker(\theta)$, then
		  $\ker(\nu) = \Core_{G}(E) = E$ which implies that $|\ker(\nu)| = p^2$. Hence, $\cod(\nu) = p^2$. Now, we prove that $\cod_{\nl}(G) = \{ p^2 \}$. 
		Just as in the case of $\phi_{3}(2111)c$ in the proof of Lemma \ref{lemma:groupsp5 phi3}, $p^2 \leq |\ker(\theta)|\leq  p^3$. If $|E| = p^3$, then owing to the fact that $\exp(A) = p^2$, we get $E\cong C_{p^2} \times C_{p}$. Then $E$ contains all the $p$-ordered elements of $A$, and hence, $G' = Z(G) \subset E$. Thus, $G'$ is contained in $\Core_{G}(E)$ as well, and hence, $E$ cannot be a candidate for $\ker(\theta)$. Therefore, $|\ker(\theta)| = p^2$. If $|E| = p^2$, then $E \ncong C_p \times C_p$ (otherwise, $G' = E$). Thus, $E \cong C_{p^2}$. By \cite[Theorem 4.2]{T2010}, total number of possibilities for such a subgroup $E$ is $p^2+ p$. In fact, $E \in \{ \langle \alpha_{1}\alpha_{2}^i \rangle, \langle \alpha_{1}^{pj}\alpha_{2} \rangle ~|~ 1 \leq i \leq p^2, 1\leq j \leq p \}$. Since $[\alpha_{1}, \alpha] = \alpha_{1}^p$ and $[\alpha_{2}, \alpha] = \alpha_{2}^p$, it is easy to check that in all the cases, $\Core_{G}(E) = E$, that is, $E$ is normal in $G$. Therefore, $|\ker(\nu)| = |\Core_{G}(\ker(\theta))| = p^2$. Hence, $\cod_{\nl}(G) = \{ p^2 \}$, and finally, $\cod(G) = \{ 1, p, p^2 \}$.\\ Using similar arguments, we conclude that $\cod(G) = \{1, p, p^2 \}$ if $ G = \phi_4(2111)a,$ or $\phi_4(1^5)$.\\
		Now, let 
		\[ G = \phi_4(221)f_0 = \langle \alpha, \alpha_1, \alpha_2, \beta_1, \beta_2 ~|~ [\alpha_1, \alpha] = \beta_1, [\alpha_{2}, \alpha] =  \beta_2 = \alpha_{1}^p,  \alpha_{2}^p = \beta_1^{\nu}, \alpha^p = \beta_1^{p} = \beta_2^{p} = 1 \rangle. \]
		Here $G' = Z(G) = \langle \alpha_1^p, \alpha_{2}^p \rangle \cong C_{p}\times C_{p}$. Take $T = \langle \alpha_1, \alpha_{2} \rangle \cong C_{p^2} \times C_{p^2}$. Then $T$ is a maximal abelian subgroup of $G$ containing $G'$. 
		Let $\phi \in \nl(G)$. Again, by \cite[Theorem 2]{BKP2013}, $\phi = \psi\ind_{T}^{G}$ for some $\psi\in \lin(T)$ with $\ker(\psi) = J$ where $J\leq T$ such that $T/J$ is cyclic. 
		If $J = \langle \alpha_1 \rangle$, then $T/J$ is cyclic and $\Core_{G}(J) = \langle \alpha_{1}^p \rangle$. If we take $J = \ker(\psi)$, then $\ker(\phi) = \Core_{G}(\ker(\psi)) = \langle \alpha_{1}^p \rangle$ which implies that $|\ker(\phi)| = p$. Hence, $\cod(\phi) = p^3$. Now, we prove that $\cod_{\nl}(G) = \{ p^3 \}$. Now, similar to the case of $\phi_4(221)c$, $J$ must be isomorphic to $C_{p^2}$. 
		Then $J$ belongs to the set $\{ \langle \alpha_{1}\alpha_{2}^i \rangle, \langle \alpha_{1}^{pj}\alpha_{2} \rangle ~|~ 1 \leq i \leq p^2, 1\leq j \leq p \}$. Since $[\alpha_{1}, \alpha] = \alpha_{1}^p$ and $[\alpha_{2}, \alpha] = \alpha_{2}^p$, it is easy to check that $\Core_{G}(J) \cong C_p$ in all the cases. That is, if $J = \ker(\psi)$, where $J$ belongs to the above-mentioned set, then $|\ker(\phi)| = p$, and thus, $\cod(\phi) = p^3$. Therefore, $\cod_{\nl}(G) = \{ p^3 \}$, and finally, $\cod(G) = \{ 1, p, p^3 \}$. 
		Using similar arguments, we conclude that $\cod(G) = \{1, p, p^3 \}$ if $G = \phi_4(221)f_r$ for $r = 1,2,\ldots, (p-1)/2$.\\
		Hence, $\cod(G) = \{ 1, p, p^2 \}$, $\{ 1, p, p^3 \}$ or $\{ 1, p, p^2, p^3 \}$ if $G\in \Phi_{4}$ with $|G| = p^5$.
	\end{proof}

	\begin{corollary} \label{lemma:groupsp5 6to10}
		Let $G$ be a group of order $p^5$ belonging to the family $\Phi_{i}$ for $i=6,7,8,10$. Then $\cod(G) = \{ 1, p, p^2, p^3 \}$.
	\end{corollary}
	\begin{proof}
		If $G\in \Phi_6$, then $\cl(G) \geq 3$, $\cd(G) = \{ 1, p \}$ and $Z(G)\cong C_p \times C_p$ (see \cite{J1980}). By Lemma \ref{lemma:p5 cl3 cdG1,p}, $\cod(G) = \{ 1, p, p^2, p^3 \}$. Finally if $G\in \Phi_6 \cup \Phi_8 \cup \Phi_{10}$, then $\cl(G) \geq 3$ with $\cd(G) = \{ 1, p, p^2\}$. Then $\cod(G) = \{ 1, p, p^2, p^3 \}$ by Lemma \ref{lemma:p5 cl3 cdG1,p,p^2}.
	\end{proof}
	
	\begin{corollary} \label{lemma:groupsp5 phi9}
		Let $G$ be a group of order $p^5$ belonging to the family $\Phi_{9}$. Then $\cod(G) = \{ 1, p, p^2, p^3, p^4 \}$.
	\end{corollary}
	\begin{proof}
		If $G\in \Phi_9$, then $\cl(G) = 4$ and $\cd(G) = \{ 1, p \}$ (see \cite{J1980}). Then $\cod(G) = \{ 1, p, p^2, p^3, p^4 \}$ by \cite[Theorem A]{CL2020 maximal}.
	\end{proof}
	
		%
	
	\begin{longtable}[c]{|l|c||l|c|}
		\caption{Character codegrees set of non-VZ $p$-groups of order $p^5$ \label{t:non VZ groups}}\\
		
		\hline
		Group $G$ &  $\cod(G)$ & 	Group $G$  & $\cod(G)$ \\ 
		\hline
		\endfirsthead
		
		\hline
		\multicolumn{4}{|c|}{Continuation of Table \ref{t:non VZ groups}}\\
		\hline
		Group $G$ &  $\cod(G)$ & 	Group $G$  & $\cod(G)$ \\   
		\hline
		
		\endhead
		
		\hline
		\endfoot

		\endlastfoot

		\vtop{\hbox{\strut $\phi_{3}(2111)a$, } \hbox{\strut $\phi_{3}(2111)b_r$ $(r=1, \nu)$, } \hbox{\strut $\phi_{3}(1^5)$, } \hbox{\strut $\phi_{3}(221)a$, } \hbox{\strut $\phi_{3}(221)b_r$ $(r=1, \nu)$, } \hbox{\strut $\phi_{3}(2111)d$, } \hbox{\strut $\phi_{3}(2111)e$}} & $\{ 1, p, p^2, p^3 \}$ & \vtop{\hbox{\strut $\phi_{3}(311)a$, }\hbox{\strut $\phi_{3}(311)b_r$ $(r=1, \nu)$, } \hbox{\strut $\phi_{3}(2111)c$ }}  & $\{ 1, p, p^2, p^4 \}$ \\
		\hline
		\hline
		\vtop{\hbox{\strut $\phi_{4}(221)a$, } \hbox{\strut $\phi_{4}(221)b$, } \hbox{\strut $\phi_{4}(221)d_r$ $(r=1,2,\ldots,(p-1)/2)$, } \hbox{\strut $\phi_{4}(221)e$, } \hbox{\strut $\phi_{4}(2111)b$, } \hbox{\strut $\phi_{4}(2111)c$, } \hbox{\strut $\phi_{4}(2111)e$}} & $\{ 1, p, p^2, p^3 \}$ & \vtop{\hbox{\strut $\phi_{4}(221)c$, }\hbox{\strut $\phi_{4}(2111)a$, } \hbox{\strut $\phi_{4}(1^5)$ }}  & $\{ 1, p, p^2 \}$ \\
		\hline
		\vtop{\hbox{\strut $\phi_{4}(221)f_0$, }\hbox{\strut $\phi_{4}(221)f_r$ $(r=1,2,\ldots,(p-1)/2)$ } }  & $\{ 1, p, p^3 \}$ & - & - \\
		\hline
		\hline
		All Groups in $\Phi_6 \cup \Phi_7 \cup \Phi_8  \cup \Phi_{10}$ & $\{ 1, p, p^2, p^3 \}$ & - & -\\
		\hline
		\hline
		All Groups in $\Phi_9$ & $\{ 1, p, p^2, p^3, p^4 \}$ & - & -\\
		\hline
	\end{longtable}
	
	\subsection{Verification of results through {\sc Magma}}
	We have constructed a {\sc Magma} function {\tt CharacterCodegree} to verify $\cod(G)$ for a given group $G$. With its help, we verify the data recorded in Table \ref{t:p34}, \ref{t:VZ groups phi2 phi5} and \ref{t:non VZ groups} for primes 5, 7, 11 and 13. The code 
	is publicly available in {\sc Magma} 
	via a GitHub repository \cite{githubudeep2026}. 
	
	\subsection{Character codegree sets for groups of order $3^n$ where $n=3,4,5$} We utilize the {\sc Magma} code {\tt CharacterCodegree} to obtain the character codegrees of non-abelian groups of order $3^n$ (for $n=3,4,5$); we present the values in Table \ref{t:3-groups}. The group identifiers are with respect to 
	the {\sc SmallGroups} library \cite{SmallGroups}.
	
	\begin{small}
		
		\begin{longtable}[c]{|l|l|}
			\caption{The character codegrees set for the non-abelian groups of order $3^n$ (for $n=3,4,5$) \label{t:3-groups}}\\
			
			\hline
			$\cod(G)$ & SmallGroup ID\\
			\hline
			\endfirsthead
			
			\hline
			\multicolumn{2}{|c|}{Continuation of Table \ref{t:3-groups}}\\
			\hline
			$\cod(G)$ & SmallGroup ID\\
			\hline
			\endhead
			
			\hline
			\endfoot
			$\{ 1, 3, 9 \}$  &   [27, 3], [27, 4] \\  
			\hline 
			\hline
			$\{ 1, 3, 9 \}$  &   [81, 3], [81, 4], [81, 12], [81, 13] \\  \hline
			$\{ 1, 3, 27 \}$  &   [81, 14] \\  
			\hline 
			$\{ 1, 3, 9, 27 \}$  &   [81, 6], [81, 7], [81, 8], [81, 9], [81, 10] \\  
			\hline \hline
			$\{ 1, 3, 9 \}$  &   [243, 2], [243, 32], [243, 33], [243, 37], [243, 38], [243, 47], [243, 62], [243, 63] \\  \hline
			$\{ 1, 3, 27 \}$  &  [243, 44], [243, 45], [243, 64], [243, 65], [243, 66]  \\  
			\hline 
			$\{ 1, 3, 9, 27 \}$  & \vtop{\hbox{\strut [243, 3], [243, 4], [243, 5], [243, 6], [243, 7], [243, 8], [243, 9], [243, 11] }  \hbox{\strut [243, 12], [243, 13], [243, 14], [243, 15], [243, 17], [243, 18], [243, 21], [243, 22]  } \hbox{\strut [243, 28], [243, 29], [243, 30], [243, 34], [243, 35], [243, 36], [243, 39], [243, 40]  } \hbox{\strut [243, 41], [243, 42], [243, 43], [243, 46], [243, 49], [243, 51], [243, 52], [243, 53]  } \hbox{\strut [243, 54], [243, 56], [243, 57], [243, 58], [243, 59], [243, 60], [243, 52], [243, 53]  } }   \\
			\hline
			$\{ 1, 3, 9, 81 \}$ &  [243, 16], [243, 19], [243, 20], [243, 50], [243, 55]   \\
			\hline
			$\{ 1, 3, 9, 27, 81 \}$  &  [243, 24], [243, 25], [243, 26], [243, 27]  \\
			\hline
		\end{longtable}
	\end{small}
	
	We summarize the results obtained in Section \ref{sec:groups less than p6} in Theorem \ref{thm:codG total p^5}. 
	\begin{theorem} \label{thm:codG total p^5}
		Let $G$ be a non-abelian $p$-group, where $p$ is an odd prime.
		\begin{enumerate}
			\item If $|G| = p^3$, then $\cod(G) = \{ 1, p, p^2 \}$.
			\item If $|G| = p^4$, then $\cod(G) \in \{ \{1, p, p^2\}, \{1, p, p^3\}, \{1, p, p^2, p^3\}  \}$.
			\item If $|G| = p^5$, then $\cod(G) \in \{ \{1, p, p^2\}, \{1, p, p^3\}, \{1, p, p^2, p^3\}, \{1, p, p^2, p^4\}, \{1, p, p^2, p^3, p^4\}  \}$.
		\end{enumerate}
	\end{theorem}
	
	\section*{Acknowledgements}
	The author thanks Sunil Kumar Prajapati for carefully reading the many versions of this paper and elaborate discussions on the content.
	The author also acknowledges the University of Auckland
	for providing remote access to their computational facilities.

\end{document}